\documentclass[preprint]{elsarticle}
\usepackage{lipsum}
\makeatletter
\def\ps@pprintTitle{%
 \let\@oddhead\@empty
 \let\@evenhead\@empty
 \def\@oddfoot{\centerline{\thepage}}%
 \let\@evenfoot\@oddfoot}
\makeatother

\usepackage{lineno,hyperref}
\usepackage{pdfsync}
\usepackage[english]{babel} 
\usepackage{textcomp} 
\usepackage{color, amssymb, amsfonts, amstext, latexsym, amsthm, amsmath}
\usepackage{mathrsfs}
\usepackage{graphicx, bm, bbm, yfonts, colortbl}
\usepackage{tikz}
\usetikzlibrary{matrix}
\usepackage[shortlabels]{enumitem}
\usepackage{wasysym}

\usepackage{geometry}

\newcommand{\Var}{\mathrm{Var}}

\newtheorem{teo}{Theorem}[section]
\newtheorem{pro}{Proposition}[section]
\newtheorem{cor}{Corollary}[section]
\newtheorem{lem}{Lemma}[section]

\theoremstyle{definition}
\newtheorem{exa}{Example}[section]

\theoremstyle{remark}
\newtheorem{rem}{Remark}[section]

\modulolinenumbers[5]

\journal{Statistics \& Probability Letters}

\begin{document}

\begin{frontmatter}

\title{Exact uniform modulus of continuity \\for $q$-isotropic Gaussian random fields\tnoteref{t1}}
\tnotetext[t1]{This work was partially supported by the grant BES-2016077051 from the \textit{Ministerio de Ciencia e Innovación},
Spain}
\author{Adrián Hinojosa-Calleja}
\ead{hinojosa.a@gmail.com}
\address{Facultat de Mamtmàtiques i Informàtica, Universitat de Barcelona, Gran Via de les Corts Catalanes, 585, 08007 Barcelona, Spain}

\begin{abstract}
We find sufficient conditions for the existence of an exact uniform modulus continuity for the class of $q$-isotropic Gaussian random fields introduced in \cite{hinojosa2021anisotropic}. We apply the result to a $d$-dimensional version of the $B^{\gamma}$ Gaussian processes defined in \cite{mocioalca2005skorohod}.
\end{abstract}


\end{frontmatter}

\section{Introduction}

A gauge function is a strictly increasing continuous function $q:[0,T]\rightarrow\mathbb{R}_+$, $T>0$ satisfying $q(0)=0$. Fix $K$ a compact set of $\mathbb{R}^d$ and assume that $X=\{X(x), x\in K\}$ is a real valued Gaussian random field. We say that $X$ is \textit{$q$-isotropic} on $K$ if there exists a gauge function $q$ and positive finite constants $c,C$ that 
\begin{equation}\label{can}
cq(\vert x-\bar{x}\vert)\leq\mathfrak{d}_{x,\bar{x}}\leq Cq(\vert x-\bar{x}\vert), x,\bar{x}\in K,
\end{equation}
where $\mathfrak{d}_{x,\bar{x}}:=\Vert X(x)-X(\bar{x})\Vert_{L^2(\Omega)}$ is the canonical metric of $X$. If $X$ only satisfies the upper bound in \eqref{can} we say that it is \textit{$\hat{q}$-isotropic}. The simplest form of a gauge function $q$ is 
\begin{equation}\label{exgau}
q(\tau)= \tau^{\nu},\tau,\nu>0. 
\end{equation}
If $X$ is a $q$-isotropic Gaussian random field with $q$ as in \eqref{exgau} it is referred as \textit{isotropic}.

Theorem 4.1 in \cite{meerschaert2013fernique} establishes general criteria for Gaussian anisotropic processes (including the case of Gaussian isotropic processes) to have an exact uniform modulus of continuity. This result or similar approaches has been applied to the study of several Gaussian processes, e.g. the fractional Brownian sheet \cite{ayache2005asymptotic}, the stochastic heat equation \cite{tudor2017sample} and the stochastic wave equation \cite{lee2019local}. 

Recently in the literature, a variety of $q$-isotropic Gaussian random fields that are not simply isotropic (i.e. with $q$ different from \eqref{exgau}) has arisen \cite{herrell2020sharp},\cite{hinojosa2022hitting}, \cite{hinojosa2022linear},\cite{sanz2018systems}. This paper provides a first approach for studying sample path continuity properties of such kinds of processes.

In Section \ref{s2} we establish sufficient conditions for the existence of the exact uniform modulus continuity of $q$-isotropic Gaussian random fields. In Section \ref{s3} we apply such results to a $d$-dimensional version of the $B^{\gamma}$ Gaussian processes introduced in \cite{mocioalca2005skorohod}. We finish this work by suggesting open problems related to the study of solutions to stochastic partial differential equations.

\section{Exact global modulus of continuity}\label{s2}

This section aims to prove Theorem \ref{2t5.1} which provides conditions on a centered $q$-isotropic Gaussian random field (from now referred as $q$-Gaussian random field) for having an exact global modulus of continuity.

Let $X$ be a centered $\hat{q}$-Gaussian random field on a compact set $K$ of $\mathbb{R}^d$. Theorem 1.3.5 in \cite{adler2009random} implies that there exists a universal constant $C_0$ and positive random variable $\eta$ such that 
\begin{equation}\label{2e5.1}
\sup_{\genfrac{}{}{0pt}{}{x,\bar x \in K,} {\mathfrak{d}_{ x,\bar x}\leq\varepsilon}}\vert X(x)-X(\bar x)\vert\leq C_0\int_0^\varepsilon d\rho\sqrt{\log\left(\frac{\diameter_K}{q^{-1}(\rho)}\right)},\varepsilon\in(0,\eta),
\end{equation} 
where $\diameter_K:=\sup_{x,\bar{x}\in K}\vert x-\bar{x}\vert$ is the Euclidean diameter of $K$.

For any $x\in K$, we denote by $K(x_-)=\{\bar{x}\in K: \bar{x}_{l}\leq x_{l}, l=1,...,d\}$ the set of points in $K$ that are at the left of $x$. We will make use of the following local nondeterminism condition on $X$:

\begin{enumerate}[label=(LND)]
\item There exists a gauge function $q$ and a positive constant $c_1$ such that for all integers $n\geq 1$, and all $x\in K$, $x^1,...,x^n\in K(x_-),$
\begin{equation*}
E(\Var(X(x)\mid X(x^1),...,X(x^n)))\geq c_1 \sum_{l=1}^{d}\bigwedge_{ j=1}^ nq( x_l-x_l^j)^{2}.
\end{equation*}
\end{enumerate}

\begin{rem} Since $X$ has second order finite moments, by \cite{durrett2019probability}[Thm. 4.1.15],
\begin{align}\label{SLNDbis}
E(\Var(X(x)\mid X(x^1),...,X(x^n)))&= \min_{a\in\mathbb{R}^n}E\left(\left[X(x)-\sum_{j=1}^na_jX(x^j)\right]^2\right).
\end{align}
\end{rem}

We introduce a set of conditions for the gauge function $q$:

\begin{enumerate}[label=(q\arabic*)]
\item The map $\tau\mapsto q(\tau)\sqrt{-\log\tau}$ is non decreasing on a neighborhood of zero.
\item $\lim_{\tau\downarrow 0}q(\tau)\sqrt{-\log\tau}= 0. $
\item There exists a positive constant $C_1$ such that for any $\tau\in[0,T]$,
$$\int_0^{\tau} d\rho q(\rho)\left[\rho\sqrt{\log\left(\frac{T}{\rho}\right)}\right]^{-1}\leq C_1q(\tau)\sqrt{\log\left(\frac{T}{\tau}\right)}.$$
\end{enumerate}

\begin{exa}\label{ex1} We analyze the conditions above for some examples of gauge functions.
\begin{enumerate}

\item $q(\tau)= \tau^{\nu}, \tau, \nu>0$. It is not hard to prove that conditions (q1) and (q2) are valid. We have that for any $T>0$
\begin{align}\label{lop}
\lim_{\tau\downarrow 0}\frac{\int_0^{\tau} d\rho q(\rho)\left[\rho\sqrt{\log\left(\frac{T}{\rho}\right)}\right]^{-1}}{q(\tau)\sqrt{\log\left(\frac{T}{\tau}\right)}}=0,
\end{align}
implying (q3).

\item $q(\tau) = \left( \log\frac{T}{\tau}\right)^{\gamma}\tau^{\nu}, \tau\in[0,T], \nu, \gamma>0$ and $T=e^{-\frac{\gamma}{\nu}}$. It is not hard to check conditions (q1) and (q2) are valid. Similarly to \eqref{lop}, by applying l'Hôpital's rule, one can prove (q3).

\end{enumerate}

\end{exa}

We are ready to state and prove the main result of this section. We follow a similar method than the proof of Theorem 4.1 in \cite{meerschaert2013fernique} (see also \cite{lee2019local}[Thm. 3.1]).

\begin{teo}\label{2t5.1} 
Let $q$ be a gauge function satisfying (q1), (q2), and (q3) with $T=\diameter_K$. Let $X$ be a centered $\hat{q}$-Gaussian random field on a compact box $K=[a,b], a,b\in\mathbb{R}^d$ with positive Lebesgue measure and satisfying (LND). Then
\begin{equation}\label{2e1}
\lim_{\varepsilon\downarrow 0} \sup_{\genfrac{}{}{0pt}{}{x,\bar x\in K,}{\mathfrak{d}_{x,\bar{x}}\leq \varepsilon}}\frac{\vert X(x)-X(\bar x)\vert}{\mathfrak{d}_{x,\bar{x}}\sqrt{\log\left(\frac{\diameter_K}{q^{-1}(\mathfrak{d}_{x,\bar{x}})}\right)}}= C\text{ a.s.}
\end{equation}
\end{teo}

\begin{rem} Although Theorem \ref{2t5.1} is stated in terms of a $\hat{q}$-Gaussian random field, any $\hat{q}$-Gaussian random field satisfying (LND) is $q$-Gaussian (see \eqref{SLNDbis}). 
\end{rem}

\begin{proof} Since $X$ is $\hat{q}$-Gaussian random field on $K$ its covariance function is continuous on $K^2$. Then, due to \cite{marcus2006markov}[Thm. 5.3.2] $X$ has a version given by
\begin{equation}\label{kar}
\tilde{X}(x)=\sum_{j=0}^\infty\varphi_j(x)\xi_j, x\in K
\end{equation}
where $(\varphi_j)_{j\in\mathbb{N}}$ are continuous functions on $K$, $(\xi_j)_{j\in\mathbb{N}}$ is an i.i.d. standard normal random variables sequence, and the sum in \eqref{kar} converges to $\tilde{X}$ in $L^2(\Omega)$. 

Let $$\tilde{X}_n(x)=\sum_{j=n}^\infty\varphi_j(x)\xi_j, x\in K. $$
We claim that for any $n\in\mathbb{N}$,
$$L:=\lim_{\varepsilon\downarrow 0} \sup_{\genfrac{}{}{0pt}{}{x,\bar x\in K,}{\mathfrak{d}_{x,\bar{x}}\leq \varepsilon}}\frac{\vert \tilde X(x)-\tilde X(\bar x)\vert}{\mathfrak{d}_{x,\bar{x}}\sqrt{\log\left(\frac{\diameter_K}{q^{-1}(\mathfrak{d}_{x,\bar{x}})}\right)}}= \lim_{\varepsilon\downarrow 0}  \sup_{\genfrac{}{}{0pt}{}{x,\bar x\in K,}{\mathfrak{d}_{x,\bar{x}}\leq \varepsilon}}\frac{\vert \tilde X_n(x)-\tilde X_n(\bar x)\vert}{\mathfrak{d}_{x,\bar{x}}\sqrt{\log\left(\frac{\diameter_K}{q^{-1}(\mathfrak{d}_{x,\bar{x}})}\right)}}. $$
The claim implies that $L$ is measurable with respect to the tail sigma field of $(\xi_j)_{j\in\mathbb{N}}$ and thus a.s. constant. This fact together with Propositions \ref{2p5.1} and \ref{2p5.2} below implies the theorem.

Indeed, by \eqref{kar}
$$\mathfrak{d}_{x,\bar{x}}^2=\sum_{j=0}^\infty(\varphi_j(x)-\varphi_j(\bar{x}))^2, \quad x,\bar{x}\in K.$$
Define
$$\tilde{Y}_n(x)=\sum_{j=0}^n\varphi_j(x)\xi_j=\tilde{X}(x)-\tilde{X}_{n+1}(x), \quad x\in K, $$
and note that
\begin{align*}
\vert \tilde{Y}_n(x)-\tilde{Y}_n(\bar{x})\vert\leq \left(\sum_{j=1}^n\vert\xi_j \vert\right)\bigvee_{j=0}^n\vert\varphi_j(x)-\varphi_j(\bar{x})\vert\leq \left(\sum_{j=1}^n\vert\xi_j \vert\right)\mathfrak{d}_{x,\bar{x}}.
\end{align*}
The last inequality and the fact that $q$ is a gauge function yields to
$$\lim_{\varepsilon\downarrow 0}  \sup_{\genfrac{}{}{0pt}{}{x,\bar x\in K,}{\mathfrak{d}_{x,\bar{x}}\leq \varepsilon}}\frac{\vert \tilde Y_n(x)-\tilde Y_n(\bar x)\vert}{\mathfrak{d}_{x,\bar{x}}\sqrt{\log\left(\frac{\diameter_K}{q^{-1}(\mathfrak{d}_{x,\bar{x}})}\right)}}\leq \left(\sum_{j=1}^n\vert\xi_j \vert\right)\lim_{\varepsilon\downarrow 0}\frac{1}{\sqrt{\log\left(\frac{\diameter_K}{q^{-1}(\varepsilon)}\right)}}=0.$$
The claim follows from the inequality above and 
\begin{align*}
\lim_{\varepsilon\downarrow 0} \sup_{\genfrac{}{}{0pt}{}{x,\bar x\in K,}{\mathfrak{d}_{x,\bar{x}}\leq \varepsilon}}&\frac{\vert \tilde X_n(x)-\tilde X_n(\bar x)\vert}{\mathfrak{d}_{x,\bar{x}}\sqrt{\log\left(\frac{\diameter_K}{q^{-1}(\mathfrak{d}_{x,\bar{x}})}\right)}}-\lim_{\varepsilon\downarrow 0} \sup_{\genfrac{}{}{0pt}{}{x,\bar x\in K,}{\mathfrak{d}_{x,\bar{x}}\leq \varepsilon}}\frac{\vert \tilde Y_{n-1}(x)-\tilde Y_{n-1}(\bar x)\vert}{\mathfrak{d}_{x,\bar{x}}\sqrt{\log\left(\frac{\diameter_K}{q^{-1}(\mathfrak{d}_{x,\bar{x}})}\right)}}\leq\\
&L\leq\lim_{\varepsilon\downarrow 0} \sup_{\genfrac{}{}{0pt}{}{x,\bar x\in K,}{\mathfrak{d}_{x,\bar{x}}\leq \varepsilon}}\frac{\vert \tilde X_n(x)-\tilde X_n(\bar x)\vert}{\mathfrak{d}_{x,\bar{x}}\sqrt{\log\left(\frac{\diameter_K}{q^{-1}(\mathfrak{d}_{x,\bar{x}})}\right)}}+\lim_{\varepsilon\downarrow 0} \sup_{\genfrac{}{}{0pt}{}{x,\bar x\in K,}{\mathfrak{d}_{x,\bar{x}}\leq \varepsilon}}\frac{\vert \tilde Y_{n-1}(x)-\tilde Y_{n-1}(\bar x)\vert}{\mathfrak{d}_{x,\bar{x}}\sqrt{\log\left(\frac{\diameter_K}{q^{-1}(\mathfrak{d}_{x,\bar{x}})}\right)}}.
\end{align*}
\end{proof}

We prove Propositions \ref{2p5.1} and \ref{2p5.2}, which establishes conditions for a $\hat{q}$-Gaussian random field to have a global modulus of continuity with a positive upper bound, and a positive lower bound, respectively. 

\begin{pro}\label{2p5.1} Let $X$ be a centered $\hat{q}$-Gaussian random field on $K$ a compact subset of $\mathbb{R}^d$, with $q$ a gauge function satisfying (q1), (q2), and (q3) with $T=\diameter_K$. Then, 
\begin{equation}\label{2e5.2}
\lim_{\varepsilon\downarrow 0} \sup_{\genfrac{}{}{0pt}{}{x,\bar x\in K,}{\mathfrak{d}_{x,\bar{x}}\leq \varepsilon}}\frac{\vert X(x)-X(\bar x)\vert}{\mathfrak{d}_{x,\bar{x}}\sqrt{\log\left( \frac{\diameter_K}{q^{-1}(\mathfrak{d}_{x,\bar{x}})}\right)}}\leq C_0(C_1+1)\text{ a.s.,}
\end{equation}
where $C_0$ and $C_1$ are the positive constants in \eqref{2e5.1} and (q3), respectively.
\end{pro}

\begin{proof}

By (q2) and (q3) (see \cite[(5)]{tindel2004sharp}), for any $\varepsilon>0$ small enough,
\begin{align}\label{2e5.3b}
\int_0^ {\varepsilon} d\rho\sqrt{\log\left(\frac{\diameter_K}{q^{-1}(\rho)}\right)}&=\varepsilon\sqrt{\log\left(\frac{\diameter_K}{q^{-1}(\varepsilon)}\right)}+\int_0^{q^{-1}(\varepsilon)}d\rho q(\rho)\left[2\rho\sqrt{\log\left(\frac{\diameter_K}{\rho}\right)}\right]^{-1}\notag\\
&\leq (C_1+1) \varepsilon\sqrt{\log\left(\frac{\diameter_K}{q^{-1}(\varepsilon)}\right)}.
\end{align}
Let $\varepsilon_n=n^{-1}$, (q1) implies that for any $n$ big enough,
\begin{align}\label{2e5.4}
\sup_{\genfrac{}{}{0pt}{}{x,\bar x\in K,}{\varepsilon_ {n+1}\leq\mathfrak{d}_{x,\bar{x}}\leq \varepsilon_n}}\frac{\vert X(x)-X(\bar x)\vert}{\mathfrak{d}_{x,\bar{x}}\sqrt{\log\left(\frac{\diameter_K}{q^{-1}(\mathfrak{d}_{x,\bar{x}})}\right)}}\leq \sup_{\genfrac{}{}{0pt}{}{x,\bar x\in K,}{\varepsilon_ {n+1}\leq\mathfrak{d}_{x,\bar{x}}\leq \varepsilon_n}}\frac{\vert X(x)-X(\bar x)\vert}{\varepsilon_{n+1}\sqrt{\log\left(\frac{\diameter_K}{q^{-1}(\varepsilon_ {n+1})}\right)}}.
\end{align}
By \eqref{2e5.1},\eqref{2e5.3b},\eqref{2e5.4} and (q1) we deduce that for any $n$ big enough
$$\sup_{\genfrac{}{}{0pt}{}{x,\bar x\in K,}{\varepsilon_ {n+1}\leq\mathfrak{d}_{x,\bar{x}}\leq \varepsilon_n}}\frac{\vert X(x)-X(\bar x)\vert}{\mathfrak{d}_{x,\bar{x}}\sqrt{\log\left(\frac{\diameter_K}{q^{-1}(\mathfrak{d}_{x,\bar{x}})}\right)}}\leq C_0(C_1+1),$$
which implies \eqref{2e5.2}.

\end{proof}

\begin{pro}\label{2p5.2} Let $X$ be a centered Gaussian random field on a compact box $K=[a,b]$, $a,b\in\mathbb{R}^d$ with positive Lebesgue measure and satisfying (LND) for a gauge function $q$. Assume that $q$ satisfies (q1). Then, there exists a finite positive constant $c_2$ depending on $c_1$ and $K$ that 
\begin{equation}\label{2e5.5}
\lim_{\varepsilon\downarrow 0} \sup_{\genfrac{}{}{0pt}{}{x,\bar x\in K,}{\mathfrak{d}_{x,\bar{x}}\leq \varepsilon}}\frac{\vert X(x)-X(\bar x)\vert}{\mathfrak{d}_{x,\bar{x}}\sqrt{\log\left(\frac{\diameter_K}{q^{-1}(\mathfrak{d}_{x,\bar{x}})}\right)}}\geq c_2\text{ a.s.}
\end{equation}
\end{pro}
\begin{proof}
Let $\gamma=\bigwedge_{l=1}^d(a_l-b_l)>0$. For each $n\geq 1$, let $\varepsilon_n=q(2^{-n}\gamma)>0.$ For $j=0,1,\ldots, 2^n$, let $x_l^{n,j}=a_l+j\gamma 2^{-n}\in K$. (q1) implies that
\begin{align}\label{2e5.6}
\lim_{\varepsilon\downarrow 0}\sup_{\genfrac{}{}{0pt}{}{x,\bar x\in K,}{\mathfrak{d}_{x,\bar{x}}\leq \varepsilon}}\frac{\vert X(x)-X(\bar x)\vert}{\mathfrak{d}_{x,\bar{x}}\sqrt{\log\left(\frac{\diameter_K}{q^{-1}(\mathfrak{d}_{x,\bar{x}})}\right)}}&=\lim_{n\rightarrow\infty}\sup_{\genfrac{}{}{0pt}{}{x,\bar x\in K,}{\mathfrak{d}_{x,\bar{x}}\leq \varepsilon_n}}\frac{\vert X(x)-X(\bar x)\vert}{\mathfrak{d}_{x,\bar{x}}\sqrt{\log\left(\frac{\diameter_K}{q^{-1}(\mathfrak{d}_{x,\bar{x}})}\right)}}\notag\\
&\geq \liminf_{n\rightarrow\infty}J_n,
\end{align}
for
$$J_n= \max_{1\leq j\leq 2^n}\frac{\vert X(x^{n,j})-X(x^{n,j-1})\vert}{\varepsilon_n\sqrt{\log\left(\frac{\diameter_K}{q^{-1}(\varepsilon_n)}\right)}}, n\geq 1.$$
Let $C_*$ be a positive constant whose value will be determined later. Fix $n$ and write $x^{n,j}=x^j$ to simplify notations. By Lemma \ref{2l5.1} bellow
\begin{align}\label{2e5.7}
P(J_n\leq C_*)\leq\prod_{j=1}^{2^n} P\left(\frac{\vert X(x^{j})-E(X(x^{j})\vert \mathcal{F}_j )\vert}{\varepsilon_n\sqrt{\log\left(\frac{\diameter_K}{q^{-1}(\varepsilon_n)}\right)}}\leq C_*\right),
\end{align}
where $\mathcal{F}_j=\sigma (X(x^0),...,X(x^{j-1}))$. 

We claim that there exists a positive constant $C_2$ depending on $c_1,C_*,K$ such that for any $n\geq 1$  and $j=1,...,2^n$,
\begin{align}\label{2e5.8}
P\left(\frac{\vert X(x^{j})-E(X(x^{j})\vert \mathcal{F}_j )\vert}{\varepsilon_n\sqrt{\log\left(\frac{\diameter_K}{q^{-1}(\varepsilon_n)}\right)}}\leq C_*\right)\leq \exp\left(-C_2\frac{2^{-\frac{n C_*^2}{2}}}{\sqrt{n}}\right).
\end{align}
Before proving the claim, we explain why it implies the proposition. By \eqref{2e5.7} and \eqref{2e5.8},
\begin{align*}
P(J_n\leq C_*)&\leq \exp\left(-C_2\frac{2^{n(1- C_*^2/2)}}{\sqrt{n}}\right), n\geq 1.
\end{align*}
We can choose now $C_*$ to be a sufficiently small constant with $1-C_*^2/2>0$, implying that $\sum_{n=1}^\infty P(J_n\leq C_*)<\infty$. Hence, by the Borel-Cantelli lemma $\liminf_{n\rightarrow\infty} J_n\geq C_*$ and we deduce \eqref{2e5.5} by \eqref{2e5.6}.

We proceed to the proof of the claim. Indeed, by (LND),
\begin{align*}
E\left(\Var(X(x^{j})\vert \mathcal{F}_j)\right)\geq c_1\sum_{l=1}^d\bigwedge_{k=1}^{2^n}q^2(x^{2^n}_l-x^{k-1}_l)=c_1d\varepsilon_n^2.
\end{align*}
By the previous inequality,
\begin{align*}
P\left(\frac{\vert X(x^{j})-E(X(x^{j})\vert \mathcal{F}_j)\vert}{\varepsilon_n\sqrt{c_1 d\log\left(\frac{\diameter_K}{q^{-1}(\varepsilon_n)}\right)}} )\leq C_*\right)\leq P\left(\vert Z\vert \leq C_*\sqrt{\log\left(\frac{\diameter_K}{q^{-1}(\varepsilon_n)}\right)}\right)
\end{align*}
where $Z$ is a standard normal random variable. Using the inequalities,
$$P(\vert Z\vert>\tau)\geq(\sqrt{2\pi }\tau)^{-1}\exp(-\tau^2/2), \tau\geq 1\text{ and } 1-\tau\leq e^{-\tau},\tau>0,$$  we deduce that for $n$ large enough
\begin{align*}
P\left(\vert Z\vert \leq C_*\sqrt{\log\left(\frac{\diameter_K}{q^{-1}(\varepsilon_n)}\right)}\right)\leq\exp\left(-\left[C_*\sqrt{2\pi\log\left(\frac{\diameter_K}{q^{-1}(\varepsilon_n)}\right)\left(\frac{\diameter_K}{q^{-1}(\varepsilon_n)}\right)^{C_*^2}}\right]^{-1}\right).
\end{align*}
By writing the value of $\varepsilon_n$, we can prove that for $n$ big enough,
\begin{align*}
P\left(\vert Z\vert \leq C_*\sqrt{\log\left(\frac{\diameter_K}{q^{-1}(\varepsilon_n)}\right)}\right)&\leq\exp\left(-C(C_*,K)\frac{2^{-\frac{n C_*^2}{2}}}{\sqrt{n}}\right),
\end{align*}
implying \eqref{2e5.8}.
\end{proof}

We made use of the following lemma in the proof of Proposition \ref{2p5.2}:

\begin{lem}\label{2l5.1} Let $X=(X_0,X_1,...,X_n)$ be a centered Gaussian random vector. Then, for any $x>0$,
\begin{align*}
P(\max_{j=1,...,n}&\vert X_j-X_{j-1}\vert<x)\\
&\leq P(\max_{j=1,...,n-1}\vert X_j-X_{j-1}\vert< x)P(\vert X_n-E(X_n\vert \mathcal{F}_{n-1})\vert< x),
\end{align*}
where $\mathcal{F}_{n-1}=\sigma(X_0,...,X_{n-1})$.
\end{lem}
\begin{proof}
We use the following version of Anderson's inequality \cite{anderson1955integral}[Cor. 2]: Let $ Y=(Y_1, ..., Y_n) $ be a centered Gaussian random vector and assume that $A\subset\mathbb{R}^{d}$ is convex and symmetric about the origin. Then,
$$
P(Y+a\in A)\leq P( Y\in A), a \in\mathbb{R}^{d}.
$$

The proof takes some ideas from Theorem 1.1 in \cite{shao2003gaussian}. By conditioning on $\mathcal{F}_{n-1}$,
\begin{align}\label{2e5.9}
P(\max_{j=1,...,n}&\vert X_j-X_{j-1}\vert<x)=E(\bm{1}(\{\max_{j=1,...,n-1}\vert X_j-X_{j-1}\vert<x\})P(\vert X_n-X_{n-1}\vert<x\vert \mathcal{F}_{n-1})).
\end{align}
Anderson's inequality implies that,
\begin{align}\label{2e5.10}
P(\vert X_n-X_{n-1}\vert<x\vert \mathcal{F}_{n-1})&\leq P(\vert X_n-E(X_n\vert\mathcal{F}_{n-1})\vert <x\vert \mathcal{F}_{n-1}).
\end{align}
The lemma follows by \eqref{2e5.9}-\eqref{2e5.10}, since $(X_j-X_{j-1})_{j=1,...,n-1}$ and $X_{n}-E(X_n \vert \mathcal{F}_{n-1})$ are independent.

\end{proof}

\section{Exact global modulus of continuity for the $q$-Brownian sheet}\label{s3}

Let $W=(W_1,...,W_d)$ be a $d$-dimensional standard Brownian motion on $\mathbb{R}_+$, assume that $q$ is a gauge function and $q^2$ is of class $C^1$ everywhere in $(0,T]$. Define the $q$-Brownian sheet as 
\begin{equation}\label{3e1}
B^{q}(x):=\prod_{l=1}^dB_l^{q}(x_l)=\prod_{l=1}^d\int_0^{x_l}K(x_l-y_l)dW_l(y_l), t\in[0,T], 
\end{equation}
where $K=\sqrt{\frac{dq^2}{d\tau}}.$ This Gaussian process was introduced in \cite{mocioalca2005skorohod} for $d=1$. As an example of an application of Theorem \ref{2t5.1}, this section is devoted to proof Theorem \ref{3t1} which establishes a uniform modulus of continuity for $B^q$.

The next proposition is a generalization of Proposition 1 in \cite{mocioalca2005skorohod}. We provide conditions on $q$ implying that $B^q$ is a $\hat{q}$-Gaussian random field. 

\begin{pro}\label{3p1} 
Let $B^q$ the $q$-sheet defined in \eqref{3e1}. Assume that $q^2$ is of class $C^2$ in $(0,T]$, and that $\frac{dq^2}{d\tau}$ is non-increasing. Then for any  $x,y\in[0,T]^d$,
\begin{align}
E(\left[B^{q}(x)-B^{q}(y)\right]^2)\leq \left[2d q^{2(d-1)}(T)\right]q^2(\vert x-y\vert).\label{3e2}
\end{align}
\end{pro}

\begin{proof} 
By \eqref{3e1}, the triangle inequality and Ito's isometry,
\begin{align}\label{3e3}
E(\left[B^{q}(x)-B^{q}(y)\right]^2)&\leq\sum_{l=1}^d\prod_{k=1}^{l-1}q^2(x_k) E(\left[B_l^{q}(x_l)-B_l^{q}(y_l)\right]^2)\prod_{k= l+1}^{d}q^2(y_k).
\end{align}
\cite{mocioalca2005skorohod}[Prop.1] implies that for $l=1,...,d$,
\begin{align}\label{3e4}
E(\left[B_l^{q}(x_l)-B_k^{q}(y_l)\right]^2)\leq 2q^2(\vert x_l-y_l\vert).
\end{align}
We deduce \eqref{3e2} by \eqref{3e3}, \eqref{3e4}, and the fact that $q$ is a gauge function.
\end{proof}

\begin{exa}\label{ex2} We analyze the hypotheses of Proposition \ref{3p1} for the gauge functions from Example \ref{ex1}.
\begin{enumerate}
\item $q(\tau)= \tau^{\nu},\tau,\nu>0$. $\frac{dq^2}{d\tau}$ is non-increasing in $\mathbb{R}^+$ if and only if $\nu \in(0,\frac{1}{2}]$, otherwise it is increasing.

\item $q(\tau) = \left( \log\frac{T}{\tau}\right)^{\gamma}\tau^{\nu},\tau\in[0,T],\nu,\gamma>0$ and $T=e^{-\frac{\gamma}{\nu}}$. We have that
\begin{align*}
\frac{d^2g^2}{dr^2}(\tau)&=2\tau^{2(\nu-1)}\left(\log\frac{T}{\tau}\right)^{2(\gamma-1)}\left[ \log\frac{T}{\tau}\left(\nu(2\nu-1)\log\frac{T}{\tau}+\gamma(1-4\nu)\right)+\gamma(2\gamma-1)\right],
\end{align*} 
implying that $\frac{dq^2}{dr}$ is non-increasing in  a small interval $(0,\bar{T}]\subset(0,T]$ if and only if $\nu \in(0,\frac{1}{2}]$, otherwise it is increasing.

\end{enumerate}

\end{exa}

The next proposition verifies that $B^q$ satisfies the local nondeterminism condition (LND):

\begin{pro}\label{3p2}  Let $B^q$ the $q$-Brownian sheet defined in \eqref{3e1}. Fix $0<t<T$, then for any $x\in[t,T]^d$, and all $x^1,...,x^n\in [t,T]^d(x_-),$
\begin{equation}\label{3e5}
E(\Var(B^q(x)\mid B^q(x^1),...,B^q(x^n)))\geq q^{2(d-1)}(t)\sum_{l=1}^d\bigwedge_{ j=1}^ nq^2( x_l-x_l^j).
\end{equation}
\end{pro}

\begin{proof} We adapt the proof of \cite{khoshnevisan2007images}[Prop.42]. We relax the notation by writing $B$ instead of $B^q$. First, assume that $d=1$.  Let $x^1,...,x^n\in [t,T](x_-)$ with $x\in[t,T]^d$.  Without loss of generality we may and will assume that $x^1\leq x^2\leq ...\leq x^n$. By \eqref{3e1}, and the Ito's isometry, for any $a\in\mathbb{R}^n$, 
\begin{align}\label{3e6}
E\left(\left[B(x)-\sum_{j=1}^na_jB(x^j)\right]^2\right)&\geq \int_{x_n}^{x}[K(x-y)-\sum_{j=1}^na_jK(x^j-y)]^2dy\notag\\
&=\int_{x_n}^{x}K^2(x-y)dy=q^2(x-x^n).
\end{align}
\eqref{3e5} follows by \eqref{SLNDbis} and \eqref{3e6}.

Now, we assume that $d>1$. Fix $x\in[t,T]^d$ and decompose the rectangle $[0,x]$ in to the disjoint union
$$[0,t]\cup\bigcup_{l=1}^dD_l(x_l)\cup\Delta(t, x) $$
where $D_l(x)=\{y\in[0,x]: 0\leq y_i\leq t, i\neq l, t\leq y_l\leq x_l \}$ and $\Delta(t,x)$ is a union of $2^d-d-1$ rectangles of $[0,x]$. This implies that for all $ x\in[t, T]^d$, 
\begin{equation}\label{3e7}
B(x)= B(t)+\sum_{l=1}^dX_l(x)+B^{'}(t, x),
\end{equation} 
where,
$X_l(x)=\int_{D_l(x)}K(x-y)d  W(dy)$, $B^{'}(t, x)=\int_{\Delta(t, x)}K( x- y)d W(dy)$.
Since all the processes on the right-hand side of \eqref{3e7} are pairwise independent, for any $a\in\mathbb{R}^n$
\begin{align}\label{3e8}
E\left(\left[B(x)-\sum_{j=1}^na_jB(x^j)\right]^2\right)\geq\sum_{l=1}^dE\left(\left[X_l(x_l)-\sum_{j=1}^na_jX_l( x^j)\right]^2\right)
\end{align}
The proof of \eqref{3e5} finishes by a similar argument than \eqref{3e6}, using \eqref{3e1}, \eqref{3e8} and that
$$
X_l(x)=B(t,...,t,x_l,t,...,t)-B(t,...,t).
$$
\end{proof}

By Propositions \ref{3p1} and \ref{3p2}, of the hypotheses of Theorem \ref{2t5.1} are valid, implying Theorem \ref{3t1} bellow. Corollary \ref{cor} can be deduced by Examples \ref{ex1} and \ref{ex2}.

\begin{teo}\label{3t1} Let $B^q$ the $q$-Brownian sheet defined in \eqref{3e1}. Fix $0<t<T$, assume that that $q^2$ is of class $C^2$ in $(0,T]$, and that $\frac{dq^2}{d\tau}$ is non-increasing. If conditions (q1), (q2), and (q3) are satisfied, there exists a finite positive constant $C$ that
\begin{equation}\label{3mc}
\lim_{\varepsilon\downarrow 0} \sup_{\genfrac{}{}{0pt}{}{x,\bar x\in K,}{\mathfrak{d}_{x,\bar{x}}\leq \varepsilon}}\frac{\vert B^q(x)-B^q(\bar{x})\vert}{\mathfrak{d}_{x,\bar{x}}\sqrt{\log\left( \frac{\diameter_K}{q^{-1}(\mathfrak{d}_{x,\bar{x}})}\right)}}= C\text{ a.s.}
\end{equation}
\end{teo}

\begin{cor}\label{cor}
The $q$-Brownian sheet with $q(\tau) = \left( \log\frac{T}{\tau}\right)^{\gamma}\tau^{\nu},\tau\in[0,\tau_0],$ $\nu\in(0,\frac{1}{2}],$ $\gamma\geq 0,$ and $T=e^{-\frac{\gamma}{\nu}}$ satisfies the limit in \eqref{3mc}. If $\gamma=0$ the limit is valid for any $T>0$.
\end{cor}

We end this section with some open questions for further investigation. Consider the following stochastic heat equation studied in \cite{herrell2020sharp} $$\partial_t u(t,x)=\mathcal{L}u(t,x)+\dot{B}, u(0,x)=0, 0\leq t\leq T, x \in \mathbb{R}^d, $$
where $\mathcal{L}$ is the generator of a Lévy process, and B is a fractional colored noise with Hurst index $H\in(\frac{1}{2},1)$ in the time variable and spatial covariance function $f$ as in \cite{balan2008stochastic}.

Fix $t_0,M>0,$ according to \cite{herrell2020sharp}[Thm. 3.4, Rem 3.5] $u=\{u(t_0,x),x\in[-M,M]^d\}$ is a centered $q$-isotropic Gaussian process  with $$q(\tau)= \vert  \log r\vert^\beta r^{2(1 \wedge\theta)},\beta=\mathbbm{1}_{\theta =1}, $$
where $\theta$ is a positive parameter that depends on $d, H$, and $f$. Furthermore, if $\theta\leq 1$, $u$ satisfies (LND) since there exist a positive constant such that for any $x,x^1,...,x^n\in [-M,M]^d$,
\begin{equation}\label{2e2}
E(\Var(u(t_0,x)\mid u(t_0,x^1),...,u(t_0,x^n)))\geq c \bigwedge_{ j=1}^ n \vert x-x^j\vert^{2\theta}.
\end{equation}

As it is mentioned in \cite{herrell2020sharp}, an open problem is to establish optimal bounds for the conditional variance when $\theta=1$, since the lower bound in \eqref{2e2} is smaller than the value of the gauge function $q$ due to the appearance of a logarithmic term. 

It is expected that similar difficulties will arise from the study of the solution to the following linear stochastic partial differential equations: The Poisson equation driven by white noise \cite{sanz2018systems}[Lem. 5.5], \cite{hinojosa2022hitting}[Thm. 2.2], the bilinear heat equation driven by white noise \cite{hinojosa2022linear}[Prop. 3.2], and the generalized fractional kinetic equation driven by time fractional-noise \cite{sheng2022hitting}[Prop 3.2].

\section*{Acknowledgments}
The author wants to thank Marta Sanz-Solé who encouraged him for writing this paper.

\bibliography{mybibfile}

\end{document}